\documentclass[11pt]{amsart}
\usepackage{amsmath,amssymb,amsbsy,amsfonts,amsthm,latexsym,amsopn,amstext,amsxtra,epic, euscript,amscd,indentfirst,verbatim,graphicx, hyperref}
\usepackage[margin=1.2in]{geometry}
\begin{document}

\newtheorem{theorem}{Theorem}
\newtheorem*{theorem*}{Theorem}
\newtheorem{conjecture}[theorem]{Conjecture}
\newtheorem*{conjecture*}{Conjecture}
\newtheorem{proposition}[theorem]{Proposition}
\newtheorem{question}[theorem]{Question}
\newtheorem{lemma}[theorem]{Lemma}
\newtheorem*{lemma*}{Lemma}
\newtheorem{cor}[theorem]{Corollary}
\newtheorem*{obs*}{Observation}
\newtheorem{condition}{Condition}
\newtheorem{definition}{Definition}
\newtheorem{proc}[theorem]{Procedure}
\newcommand{\comments}[1]{} 
\def\Z{\mathbb Z}
\def\Za{\mathbb Z^\ast}
\def\Fq{{\mathbb F}_q}
\def\R{\mathbb R}
\def\N{\mathbb N}
\def\C{\mathbb C}
\def\k{\kappa}
\def\grad{\nabla}

\title[$k$-Gauduchon Complex Geometry]{On the Hermitian Geometry of $k$-Gauduchon Orthogonal Complex Structures}
 \author[University of Michigan]{Gabriel Khan} 
 
\email{gabekhan@umich.edu}

\date{\today}

\maketitle

\begin{abstract}

The purpose of this note is to study the complex structures orthogonal to a given Riemannian metric. For another paper on this topic, we highly recommend the work of Salamon \cite{Salamon1}. His work describes in great detail the role that curvature plays in this question. We instead focus on torsion, which lends itself to somewhat different analysis of the problem. 
In terms of novel results, we show that for a certain range of $k$, the spaces of $k$-Gauduchon orthogonal complex structures are pre-compact within the moduli space of complex structures and give a preliminary picture of how certain special complex structures fit into the space of orthogonal complex structures.
\end{abstract}

\section{Introduction}

Given a Riemannian manifold $M^{n}$, an {\em almost complex structure} is a smoothly varying endomorphism of $TM$ (denoted $J$) which satisfies $J^2 = -Id$. Intuitively, $J$ corresponds to the action of multiplication by $i$. We say that an almost complex structure is {\em compatible} with, or {\em orthogonal} to a metric, if it satisfies $g(X,Y) = g(JX,JY)$ for all tangent vectors $X$ and $Y$. Intuitively, this corresponds to the fact that multiplication by $i$ does not change the norm of a vector. In this case, the triple $(M, g, J)$ is called an almost-Hermitian manifold. From the definition of an almost complex structure, it follows that an almost complex manifold must be even dimensional and be oriented, but there are other restrictions as well.

For $(M,g,J)$ to be genuinely complex, there needs to be be a local biholomorphism between our manifold and $\mathbb{C}^{n/2}$. In other words, we want there to be local holomorphic coordinates, in which case we say the almost complex structure is {\em integrable}.
Integrability is a fairly restrictive condition; given an arbitrary almost complex structure $J$, such coordinates may not exist. For instance, there is an orthogonal almost complex structure on the round six-sphere, but it is not integrable. It is an open question whether any $\mathbb{S}^6$ admits a complex structure at all, but there is no complex structure which is close to being orthogonal to the round metric \cite{Salamon1}. Furthermore, no sphere other than $S^2$ or $S^6$ admits any almost complex structures.

Determining when an almost complex structure is integrable at first seems a difficult task, but the Newlander-Nirenberg theorem gives a simple criteria. It states that an almost complex structure is integrable if and only if its {\em Nijinhuis tensor} vanishes \cite{NN}.

The Nijenhuis tensor is defined in the following way, and can be interpreted as the anti-holomorphic part of the Lie bracket of two holomorphic vector fields:
  \[ N_J(X,Y) = -J^2[X,Y] + J [X,JY] + J[JX,Y] - [ JX,JY]. \]

One direction of this theorem is straightforward. Whenever $J$ is induced from complex coordinates, a computation shows that the Nijinhuis tensor must vanish. The other direction is more difficult, and requires some PDE analysis in order to prove.
 
 In fact, it is somewhat of a misnomer to refer to the Newlander-Nirenberg theorem as a singular theorem, as it more accurately describes a collection of results. The result was proven for analytic almost complex structures by Newlander and Nirenberg. The theorem was then extended to $C^{1,\alpha}$ almost complex structures by Malgrange \cite{BM}. For many years, the $C^{1,\alpha}$ result was state of the art. However, recently the Newlander-Nirenberg theorem has been proven in even weaker regularity. In \cite{DMT}, Hill and Taylor showed that the Nijinhuis tensor is well defined and that the Newlander-Nirenberg theorem holds for complex structures which are $C^\alpha$ for $\alpha > 1/2$.


\subsection{The Levi-Civita and Chern Connections}

In order to study the problem of orthogonal complex structures from the perspective of torsion, it is necessary to introduce the notion of torsion. To do so, we discuss two canonical connections on any complex manifold

On any Hermitian manifold, there are many connections of interest (for other examples, see \cite{YZ2}). We will focus on two such connections, the {\em Levi-Civita} and {\em Chern} connection.

 The Levi-Civita connection is the unique metric (i.e. $\nabla g = 0$) connection which is torsion-free. We denote this connection as $\nabla$. In general, it is not the case that $\nabla J \equiv 0$. However, if $J$ is parallel with respect to $\nabla$, we say that the complex structure is {\em K\"ahler}, and this is an important class of complex structures. There are many equivalent ways to define the K\"ahler condition, such as that $J$ is integrable and that the {\em K\"ahler form } $\omega = g(J \cdot, \cdot)$ is closed.


The Chern connection is the unique connection which is metric, complex (i.e. $\nabla J = 0$), and whose torsion is of type $(2,0)$ with respect to $J$. We denote this connection as $\nabla^{c}$. This connection can be defined using the following formula \cite{ABD}: 
\[ g(\nabla^{c}_X Y, Z) = g(\nabla_X Y, Z) - \frac{1}{2} d\omega (JX, Y, Z). \] 

These two connections are the same if and only if the complex structure is K\"ahler. We are primarily interested in the non-K\"ahler case, as much is known about K\"ahler complex structures and K\"ahler geometry in general.

\subsection{The Torsion of the Chern Connection}

 Since the Chern connection is metric, it differs from the Levi-Civita only by a contorsion tensor. Recall that the torsion is defined in the following way.
 
  \[ T^c (X,Y) = \nabla^c_X Y - \nabla^c_Y X - [X,Y] \] 

Much of the geometric information about the complex structure is contained within the torsion, so it is of interest to try to understand this tensor.

 By taking the trace of the full torsion tensor, one can define the torsion one-form $\eta$. In other words, for any unitary frame $\{ e_j \}$, we define  $\eta(X)$ as $\sum_{j=1}^n T^c{}_{Xj}^j$. It is worth noting that the Lie form, which is usually defined as $\theta= J \delta \omega$, is directly related to $\eta$. If we let $\eta^{(1,0)}$ denote the holomorphic part of $\eta$, we have that $\theta_i = -2 \eta_i^{(1,0)} - 2 \overline{ \eta_i{(1,0)}}$ \cite{YZ}.

One possible relaxation of the K\"ahler condition is to insist that the torsion one-form vanishes, in which case the complex metric is referred to as {\em balanced}. This condition was first introduced by Michelson \cite{MLM} and has attracted attention because of its role in non-K\"ahler string theory \cite{FY}.

From a geometric standpoint, curvature is far better understood than torsion. However, torsion is well suited for our analysis because the torsion of the Chern connection controls $\| \nabla J \|$.

%


\begin{lemma*}
There is a $C^1$ estimate on the complex structure in terms of the torsion. 
\[ |\nabla^{L.C.} J| \leq 3 |T^{c}| \]
\end{lemma*}

To see this, denote $\tau_{ijk} := \sum_{l=1}^n  g_{kl} T^{c}{}_{ij}^l$ and note that the following formula holds:
$$- \frac{1}{2} d\omega (JX, Y, Z)= \frac{1}{2}( \tau_{ JX \, Y \, Z}  -\tau_{  Y \, Z  \, JX} + \tau_{ Z \, JX \, Y}) $$

 This observation motivates the need to understand torsion, and it is of interest to determine conditions which control the torsion. We will focus on one such condition, which is that the complex structure is {\em $k$-Gauduchon}.


\subsection{$k$-Gauduchon complex structures and the Gauduchon curvature}

Originally introduced by Fu, Wang, and Wu \cite{FWW}, a K\"ahler form is said to be {\em $k$-Gauduchon } when the following identity holds.
 \[ i \partial \bar \partial (\omega^k) \wedge \omega^{n-k-1}=0. \]
 When $k=n-1$, this is the condition $\partial \bar \partial \omega^{n-1} = 0$, which is more commonly known as Gauduchon, or standard Gauduchon. Gauduchon showed that for any complex structure $J$ and compatible metric $g$, there is a metric $\tilde g$ conformal to $g$ such that $\tilde \omega = \tilde g(J \cdot, \cdot)$ is Gauduchon \cite{PG}. 

The motivation for our consideration of $k$-Gauduchon complex structures is the following calculation, which establishes control over the torsion for $k$-Gauduchon complex structures.

 A result due to Gauduchon \cite{PG} shows that for any integrable complex structure, the Lie form $\theta$ satisfies the following identity:
\begin{equation} \label{Gauduchon identity}
 |\theta|^2 + 2 \delta \theta =  \frac{2n-2}{2n-1} S - 2   \langle  W ( \omega ) , \omega \rangle 
 \end{equation}

In this equation, $S$ is the scalar curvature, and $W$ is the Weyl curvature, viewed as a map of two forms. Using this, we can define the Gauduchon curvature $Gaud$ as follows:
$$Gaud = \frac{2n-2}{2n-1}S - \langle W(\omega), \omega \rangle.$$

After manipulating this equation using the $k$-Gauduchon condition, we find the following equation for any $k$-Gauduchon complex structure.

 \[ (n-2)Gaud = (2k-n) |\eta|^2 + (n-k-1)| \tau |^2 \]

Whenever $(2k-n), (n-k-1) >0$, this gives us point-wise bounds on torsion in terms of Riemannian curvature alone. The Gauduchon curvature is bounded by the scalar curvature and the smallest eigenvalue of the Weyl tensor acting on two forms. Therefore, we define the $G$-curvature in the following way.

\begin{equation} \label{G curvature}
 G(x) := \sup_{\alpha \in \bigwedge^2 T^*_x M } \left( \frac{2n-2}{2n-1}S(x) - \langle W(\alpha), \alpha \rangle \right)
\end{equation}

By definition, $Gaud \leq G$, and so for a $k$-Gauduchon complex structure,

 \[ (2k-n) |\eta|^2 + (n-k-1)| \tau |^2 \leq (n-2) G. \]

 For more details, a derivation of this was done in \cite{GK}.

\subsection{Weakly $k$-Gauduchon complex structures}
We can extend the notion of $k$-Gauduchon to $C^1$ complex structures, using the previous identity. We say that a complex structure is {\em weakly $k$-Gauduchon} if \[ (n-2)Gaud = (2k-n) |\eta|^2 + (n-k-1)| \tau |^2. \] This definition is equivalent to the standard one for twice-differentiable complex structures, but also holds for of extending once-differentiable complex structures. For weakly $k$-Gauduchon complex structures, the torsion is again bounded by the $G$-curvature. As such, there is a uniform bound on $\| \nabla J \|$. 

\section{The Moduli of Orthogonal Complex Structures}

A central questions in complex geometry is to understand the moduli space of complex structures on a given manifold. Kodaira-Spencer theory shows that the moduli space of complex structures is locally a finite-dimensional analytic space. However, this space may have singularities and its global geometry can be very complicated. We do not expect to have a full picture of the moduli space, except in some special cases.

Instead of considering this general question, we study in moduli of complex structures orthogonal to a given metric. This space may also be complicated and may not have the structure of an analytic space. However, in certain special cases, it is much simpler than the full moduli space. We start with several test cases, to motivate our work and suggest what may be true more generally.



\subsection{Riemann surfaces and Teichm\"uller theory}

For Riemann surfaces, the study of the moduli of complex structures is known as Teichm\"uller theory \cite{JH}. Teichm\"uller theory is a rich field of study; the moduli space has interesting geometry in its own right. A more thorough discussion of this topic would lead us away from our original project, so we will not spend much time discussing Teichm\"uller theory except to note the following facts.

First, given a Riemannian metric on a surface, there is a unique complex structure orthogonal to that metric. As such, the study of complex structures orthogonal to a given metric is rather uninteresting. However, any conformal metric admits the same orthogonal complex structure, so to understand the space of complex structures, we should really consider the set of conformal classes. In every conformal class, there is a unique unit volume metric of constant curvature. As such, these form a natural choice of metrics from which to study the moduli space of complex structures.

For each complex structure, there is a unique metric of constant curvature orthogonal to it. Therefore, there is a bijection between the moduli space of complex structures on a surface and the moduli space of constant curvature metrics on a surface. From this observation, one can start to study the dynamics and geometry of this moduli space. Trying to go in more detail would take us too far from our original project, so we refer to the interested reader to \cite{CM}.

From the simple correspondence between constant scalar curvature metrics and complex structures in the surface case, it seems natural to try to repeat the analysis in higher dimensions and see what holds true.

 \subsection{An example in higher dimensions}

 In higher dimensions, the situation is far less well understood. It turns out that there are some important examples that can be analyzed in detail. These spaces help gain insight into how rich the space of orthogonal complex structures can be. In this section, we will focus on a single example with very rich and interesting geometry. We will consider orthogonal complex structures on the flat six-dimensional tori, which we think of as $\mathbb{R}^6$ modulo a lattice. 

\subsubsection{The K\"ahler complex structures}

Given any flat torus, there is a moduli of $SO(2n)/U(n)$ of K\"ahler complex structures orthogonal to the metric. To construct these complex structures, one considers the set of rotations of $R^{2n}$ and mods out by the ones that induce the same complex structure.

For our example, this induces a $SO(6)/U(3) \cong \mathbb{CP}^3$ family of K\"ahlerian structures on each flat torus \cite{BSV}. As such, we immediately see that the orthogonal complex structures are no longer unique, as in the surface case.

If we vary the flat metric by choosing a different lattice in $\mathbb{R}^6$ to quotient by, then for each choice of lattice, there is a family of K\"ahlerian complex structures.
Due to the following theorem of Baus and Cortes \cite{BC}, any K\"ahler structure on a topological torus is biholomorphic to a standard flat torus. 

\begin{theorem*} Let $X$ be a compact complex manifold which is aspherical with virtually
solvable fundamental group and assume $X$ supports a K\"ahler metric. Then
$X$ is biholomorphic to a quotient of $\mathbb{C}^n$ by a discrete group of complex isometries.
\end{theorem*}

As such, if we consider the space of flat unit-volume metrics fibered by the K\"ahler structures orthogonal to them, then this space surjects onto the space of all K\"ahler structures on a torus. Even stronger, in the large of a complex tori, all of the complex structures are again complex tori \cite{FC}. Therefore, this fiber bundle covers a connected component of the Kodaira-Spencer moduli space and we cannot deform a complex torus to a non-K\"ahler complex manifold.

 Naively, we might hope that this surjection gives the moduli space of complex tori the structure of a $SO(6)/U(3)$ fiber-bundle, in which case there would be a close parallel to the picture in Teichm\"uller theory. Unfortunately, this isn't the case. When one considers the space of complex tori as the space of complex lattices modulo those which induce the same complex torus (i.e. $GL_n(\mathbb C) \backslash GL_{2n}(\mathbb R) / GL_{2n}(\mathbb Z)$), then an observation of Siegel shows that the associated quotient space is not even Hausdorff. As such, the desired fiber bundle structure cannot exist.
 
 It is possible to see more concretely how the fiber bundle structure on flat metrics fails to induce a fiber bundle structure on the moduli of complex tori. Given an elliptic curve $E$ and a flat $4$-torus $T^4$, then for any constant $\alpha \neq 1$, the flat torus $E \times T^4$ is distinct from the flat torus $\alpha E \times \frac{1}{\sqrt{\alpha}}T^4$ as Riemannian metric. However, both of these spaces share the complex structures which act on each factor.

\subsubsection{The non-K\"ahler complex structures}

There is another complication to the story of complex structures on the torus in the existence of non-K\"ahler complex structures. This case is important in that we know a full classification of all possible non-K\"ahler structures \cite{KYZ}, which we will use here.

 When the flat metric is a product of an elliptic curve and a four dimensional torus, there are non-K\"ahler ``warped product" complex structures. These examples were originally discovered by Blanchard \cite{Blanchard} and Sommese \cite{Sommese} and studied in the context of twistor spaces by Borisov, Salamon and Viaclovksy \cite{BSV}. All non-K\"ahler complex structures on a flat torus are of this form, so If the metric is not a product of an elliptic curve and a four-dimensional torus, then all orthogonal complex structures are K\"ahler \cite{KYZ}.

For a given product metric, the moduli of non-K\"ahler structures is isomorphic to the space of doubly-periodic maps from the elliptic curve $E$ to the projective plane. This result shows that there are infinitely many connected components of orthogonal complex structures (depending on the degree of the map) but that each connected component orthogonal to a given metric is a compact smooth space.

Before moving on, we should note that not much is known about complex structures on the torus which are not orthogonal to flat metrics. By the earlier theorem of Baus and Cortes, any such example would have to be non-K\"ahler. However, there are currently no known examples. 

Without considering the metric, Catanese has studied the problem and posed some interesting open questions. In particular, he asks whether a complex manifold diffeomorphic to a torus with trivial canonical bundle is biholomorphic to a torus \cite{FC04}. If the answer is affirmative, this implies that the all complex structures on a six torus with trivial canonical bundle are orthogonal to a flat metric.

\subsection{A More General Picture}

From studying flat metrics, it is tempting to think that orthogonal complex structures are in abundance. In fact, the precise opposite is generally true.

\begin{obs*}
Generically, a Riemannian metric does not admit an orthogonal complex structure, and the obstruction is a point-wise generic condition.
\end{obs*}

This is a consequence of Gray's identity, which shows that the Weyl tensor must have small rank for a metric to admit an orthogonal complex structure \cite{Gray} (see also \cite{YZ}). In fact, there are many curvature conditions which are known to rule out the existence of orthogonal complex structures (for example, see \cite{LHL}). For $4$ manifolds, Salamon studied the curvature restriction in great detail. His work shows that if a pair of orthogonal complex structure exists on a $4$-fold ($J$ and $-J$), generically they are unique. He also showed that if more than two pairs exists, then infinitely many orthogonal complex structures do and the metric is locally anti-self-dual \cite{Salamon1}. All these results suggest that we should expect the richest moduli for special metrics, which we restate in terms of the following conjecture.

 
 \begin{conjecture}
  The dimension and number of complex structures orthogonal to a given metric are upper-semicontinuous functions in terms of the Riemannian metric.
 \end{conjecture}

 Our earlier analysis shows that this conjecture is true for flat metrics on the six-torus, so it remains to show it when the metric is not flat. 
 For general $4$-folds, Salamon's work is nearly enough to prove that the number of orthogonal complex structures is upper-semicontinuous. His result shows that the number of local orthogonal complex structures in a neighborhood of each point is upper-semicontinuous, but there can be global obstructions to integrability, as well.
 
 This conjecture has a corresponding version in general deformation theory, where the dimension of the space of holomorphic deformations of a fiber is upper semi-continuous in the Zariski topology. This upper semi-continuity holds in much greater generality and for a more complete version, we refer the reader to III.11 of Hartshorne \cite{hartshorne}.

There are two things that should be pointed out here. Firstly, the conjectured upper semi-continuity does not hold if one allows the metric to collapse (for instance, it fails if one collapses a 3-dimensional torus to an elliptic curve), so it is necessary to assume the metrics are non-collapsing. Secondly, we do not quotient out holomorphic isometries in the statement of this conjecture; such complex structures are considered distinct.

\subsection{Stability under deformations}

One of the key results of Kodaira-Spencer theory is that a small deformation of a K\"ahler complex structure is K\"ahler \cite{KS}. However, for non-K\"ahler complex structures, determining which properties are deformation stable is often a difficult problem. 

One open question is whether the deformation of a BSV-tori is a BSV-tori. Due to the following result of Catanese \cite{FC}, we know that BSV-tori are deformation open.

\begin{theorem*}[\cite{FC} 5.9]
A small deformation of a Sommese-Blanchard 3-fold with large enough degree is again a Blanchard-Calabi 3-fold.
\end{theorem*}

In our context, this immediately imples that a small deformation of a BSV-tori is again a BSV-tori. This does not show that BSV-tori are deformation closed, so does not answer the global question. However, if one instead considers torus bundles where over a curve of genus $g \geq 2$ instead of an elliptic curve, then all deformations of these spaces are of the same type \cite{FC04}, so a similar result is known for a similar family of complex three-folds.

For special complex structures, various results are known about the stability properties. For instance, the balanced condition ($d \omega^{n-1}=0$) is not necessarily stable under deformations \cite{AB} but if a weaker version of the $\partial \bar \partial$ lemma holds, the balanced condition is stable \cite{AU}. Similarly, Popovici showed that the strongly Gauduchon property ($ \partial \omega^{n-1}$ is $ \bar \partial$ exact) is open under holomorphic deformations \cite{DP}. The standard Gauduchon (i.e. $k=n-1$) condition is also an open condition as all complex structures are Gauduchon for some compatible metric. It remains an open question as to whether the $k$-Gauduchon condition is open under deformation for other $k$.

It is worth noting that orthogonality to a metric is closed and not open condition.  Closure can be seen since $D_J(X,Y)= g(X,Y)-g(JX,JY)$ is continuous in $J$ for any choice of smooth vector fields $X$ and $Y$. For a simple example that shows that it is not open, the deformation of an elliptic curve is not generally compatible with the original metric.

\subsection{Torsion estimates}

In order to study orthogonal complex structures, a natural approach is to try to estimate the torsion in terms of the Riemannian geometry alone. Unfortunately, this is not possible. The example of the BSV-tori show that we can find a complex structure on a flat metric whose torsion tensor is arbitrarily large in $L^2$-norm. If we let $f$ be the doubly periodic map from the elliptic curve to $\mathbb{CP}^1$ and $v_2$ be the volume of the four-torus $M_2$, then we have the following formula \cite{KYZ}:

\begin{equation*}
\int_M |T^c|^2dv = 32\pi \, v_2 \, \text{deg}(f), \label{total torsion norm}
\end{equation*}

From this, one can make the $L^2$ norm of the torsion arbitrarily large by increasing the degree of the doubly-periodic map. These examples show that the Riemannian geometry alone does not control the torsion or the size of the space of orthogonal complex structures. It's worth noting that the BSV tori are all balanced, which shows that the balanced condition is not enough to control the full torsion tensor. 

Trying to estimate the torsion is not completely hopeless. For a general complex manifold, there is an $L^2$ estimate on the torsion one-form by integrating Gauduchon's identity for the torsion one-form \ref{Gauduchon identity} \cite{PG}. 
 Furthermore, with the stronger assumption that the complex structure is $k$-Gauduchon with $n/2<k < n-1$, Gauduchon's identity yields stronger estimates on the torsion tensor, which we can use to try to control the geometry.

\subsection{K\"ahler and $k$-Gauduchon complex structures}

One natural question that we can ask is when the space of orthogonal complex structures is compact. From the preceding discussion on flat tori, we see that it is possible for the space to be very large. However, for special complex structures, compactness results do hold.

As a simple example, we can consider the K\"ahler orthogonal complex structures. Given a compact Riemannian manifold, if there is a sequence of K\"ahler orthogonal complex structures $J_i$, then all of these structures satisfy all of the complex structures satisfy $\nabla J_i \equiv 0$. As such, we can use the Arzela-Ascoli theorem to find a subsequence which converges in $C^1$\footnote{To obtain $C^1$ convergence, have to use the fact that $\nabla^2 J_i \equiv 0$ as well}. Therefore, the limiting almost complex structure is both integrable and K\"ahler. This immediately implies the following result.

\begin{theorem*}
Let $(M^n, g)$ be a compact Riemannian manifold. The set of K\"ahler complex structures orthogonal to $g$ is compact in the space of all complex structures.
\end{theorem*}

 Without too much difficulty, this argument can be extended to a sequence of Riemannian metrics, so long as the metrics are bounded in $C^2$-sense and are non-collapsing. Since any small deformation of a K\"ahler complex structure is again K\"ahler \cite{KS}, this shows that the singularities of the moduli space occur only when the compatible metrics also become singular.  These results are in stark contrast to the work of Hironaka \cite{HH}, which showed that the K\"ahler property is not necessarily closed under holomorphic deformations. 
 
Here, the Riemannian geometry is able to detect the singularities, so even though we may be able to deform a K\"ahler complex structure to a non-K\"ahler one, we cannot do so without affecting the Riemannian geometry. With the model of K\"ahler structures as a guide, we wish to recreate this analysis for $k$-Gauduchon complex structures.

 \begin{theorem*}
Let $(M^n, g)$ be a compact Riemannian manifold. For $n/2 < k < n-1$, the set of $k$-Gauduchon complex structures orthogonal to $g$ is pre-compact in the space of all complex structures orthogonal to $g$.
\end{theorem*}

\begin{proof}

We will actually prove a slightly stronger result in terms of a sequence of Riemannian metrics, not just a single one. The key observation used is that if $J$ is $k$-Gauduchon for some $k$ with $n/2 < k < n-1$, then we have a bound on the entire torsion tensor of $J$. From the earlier lemma, this provides a bound on $ \| \nabla^{L.C.} J \|$ in terms of the curvature of $g$.

Consider a sequence of metrics $g_i$ whose injectivity radius is uniformly bounded from below and which converge to a smooth metric tensor $g_\infty$ in $C^2$ norm. Given a sequence of weakly $k$-Gauduchon complex structures $J_i$ compatible with $g_i$, this sequence is of $J_i$ is equibounded in $C^1$-norm. Therefore, the sequence of $J_i$ is compact in $C^r$ for $\frac{1}{2} < r < 1$ . By the Arzela-Ascoli theorem, $J_i$ converges in $C^r$ along a subsequence to some $J_\infty$ compatible to $g_\infty$. Since all of the $J_i$ are integrable, their Nijenhuis tensors all vanish and so the formal Nijinhuis tensor of $J_\infty$ must also vanish using the results of Denson and Taylor \cite{DMT}. This furthermore shows that $J_\infty$ is integrable.
\end{proof}

It is possible to refine our subsequence so that $J_\infty$ is Lipschitz, but the complex structures may not converge in the Lipschitz sense, so this does not show that $k$-Gauduchon structures are compact. In order to show this, we must be able to extract a subsequence that converges to an orthogonal weakly $k$-Gauduchon complex structure. The most natural ways to do this would be either to extend the definition of $k$-Gauduchon to $C^r$ complex structures for $r<1$ or show that the $k$-Gauduchon condition induces more regularity than a Lipschitz estimate.  At present, we see no way to do so, and so leave it as a question.

 \begin{question}
Let $(M^n, g)$ be a compact Riemannian manifold. For $n/2 < k < n-1$, are the set of $k$-Gauduchon complex structures orthogonal to $g$ compact?
\end{question}

 This theorem does not always hold when $k=n-1$, as shown by the BSV tori.

Before moving on, we briefly mention the case where a complex structure is $k$-Gauduchon for multiple values of $k$. If a complex structure is $k$-Gauduchon for two distinct $k_1$ and $k_2$, then it satisfies $2 \| \eta\|^2 = \| T \|^2 = 2G$. In this case, then the complex structure is immediately $k$-Gauduchon for all $k$. Non-K\"ahler examples of such complex structures were constructed by Latorre and Ugarte \cite{LU}.

\subsection{A Preliminary Picture of the Moduli}

At this point, we can present the following preliminary picture of the moduli space of orthogonal complex structures. Given a compact Riemannian manifold $(M^n,g)$, we let \{ OCS \} be the set of orthogonal complex structures. From the previous results, we have the following inclusions and properties.

\[
\textrm{ \{ K\"ahler OCS\} } \subset \textrm{ \{ $k$-Gauduchon OCS for all $k$ \} } \subset \textrm{ \{$k$-Gauduchon OCS} \}_{ k = 1,\ldots ,n-1 } 
\]
\[
\subset \{ \textrm{ OCS} \} \subset \{ \textrm{ Orthogonal  Almost Complex Structures} \}
\]

The orthogonal almost complex structures on a compact manifold may form an an infinite dimensional manifold\footnote{The set of orthogonal almost complex structures can also be empty.}. When non-empty, the orthogonal complex structures will form a locally finite dimensional subset of this larger space. It is possible that the standard Gauduchon orthogonal complex structures are exactly the same as the orthogonal complex structures, and that both have infinitely many connected components of arbitrary high dimension, so no compactness results hold in general for these two subsets. For $n/2 < k < n-1$, the set of $k$-Gauduchon orthogonal complex structures is precompact in OCS, so form a much smaller subset of complex structures. For any pair $k_1$ and $k_2$, the intersection of $k_1$-Gauduchon OCS and $k_2$-Gauduchon OCS are complex structures that are $k$-Gauduchon for all $k$. The set \{K\"ahler OCS\} forms a compact subset of \{$k$-Gauduchon OCS for all $k$ \}, but the latter may contain non-K\"ahler elements as well.

\section{Acknowledgments}

I would like to thank Fangyang Zheng for his guidance Kori Khan for her help editing. This work was partially supported by DARPA/ARO Grant W911NF-16-1-0383 (PI: Jun Zhang, University of Michigan).

\bibliography{bibfile}
\bibliographystyle{alpha}


\end{document}